\documentclass{article}

\usepackage[utf8]{inputenc}
\usepackage{amsmath}
\usepackage{amssymb}
\usepackage{amsfonts}
\usepackage{enumerate}
\usepackage{amsthm}
\usepackage{amsfonts}
\usepackage{graphicx}
\usepackage{dsfont}
\usepackage{bm}
\usepackage{mathrsfs}
\usepackage{comment}
\usepackage{graphicx}
\usepackage{geometry}
\usepackage{url}
\usepackage{nowidow}

\newtheorem{theorem}{Theorem}
\newtheorem{lemma}{Lemma}
\newtheorem{proposition}{Proposition}
\newtheorem{corollary}{Corollary}

\newtheorem{remark}{Remark}

\usepackage{graphicx}
\graphicspath{ {./figures/} }

\usepackage[
backend=biber,
style=alphabetic,
sorting=ynt
]{biblatex}
\renewbibmacro{in:}{}

\title{Effective equidistribution of rational points on  horocycle sections in ASL$(2,\mathbb{Z})\backslash \text{ASL}(2,\mathbb{R})$ } 
\author{Sam Pattison}
\date{}

\begin{document}

\maketitle

\begin{abstract}
    In this paper we prove an effective equidistribution result for both primitive and non-primitive points on certain expanding horocycle sections in $ \text{ASL}(2,\mathbb{Z}) \backslash \text{ASL}(2,\mathbb{R})$. This provides an effective version of a result recently proven by the author and generalises recent work by Einsiedler, Luethi and Shah who prove similar results in the context of the unit tangent bundle of the modular surface. 
\end{abstract}

\section{Introduction}

Let $G := $ASL$(2,\mathbb{R})$ be the semi-direct product SL$(2,\mathbb{R}) \ltimes \mathbb{R}^2$
with multiplication rule 
$$
(M,x)\cdot (M',x') = (MM',xM' + x).
$$
Let $\Gamma =  $ SL$(2,\mathbb{Z}) \ltimes \mathbb{Z}^2$ be the subgroup of $G$ consisting of elements with integer entries. With respect to the Haar measure $m_G$ on $G$, $\Gamma$ is lattice in $G$. This means the quotient space $X := \Gamma \backslash G$ can be given a Haar measure $\nu$ by restricting the Haar measure on $G$ to a fundamental domain for the left-action of $\Gamma$ on $G$. This measure $\nu$ is invariant under the action of $G$ on $X$ given by $g \cdot x = xg^{-1}$. As a result, the homogeneous space $X$ is a measure preserving dynamical system $(X,\mu,G)$. We fix a left-invariant Riemannian metric $\rho$ on $G$ 

On the space $G$, we consider the following one parameter subgroups
$$
u(t) := \bigg(\begin{pmatrix}
1 & t \\
0 & 1
\end{pmatrix}, (0,0) \bigg),
$$

$$
\Phi(t) := 
\bigg(\begin{pmatrix}
e^{t/2} & 0 \\
0 & e^{-t/2}
\end{pmatrix}, (0,0) \bigg),
$$
and set
$$
a(y) := \Phi(\log(y)) = \bigg(\begin{pmatrix}
y^{1/2} & 0 \\
0 & y^{-1/2}
\end{pmatrix}, (0,0) \bigg).
$$
Orbits of $x \in X$ under the right action of the subgroup $u(t)$ correspond, via projection to the first coordinate, to horocycles in the unit tangent bundle of the modular surface  SL$(2,\mathbb{Z}) \backslash $ SL$(2,\mathbb{R})$. Hence, $u(t)$ is an example of a \textit{horocycle section}. In general, following \cite{elkies2004gaps}, we call a horocycle section a map  $n:\mathbb{R} \to G$ which takes the form
$$
n(t) := (I,\xi(t))u(t)
$$
for a continuous function $\xi:\mathbb{R} \to \mathbb{R}^2$.

For such a horocycle section, we define $\Lambda(t) = \Lambda_{\xi}(t) := t \xi_1(t) + \xi_2(t)$. Such a horocycle section is called \textit{periodic}, with period $P \in \mathbb{N}$, if there exists $\gamma \in \Gamma$ such that $n(t+P) = \gamma n(t)$ for all $t \in \mathbb{R}$. Moreover, we call such a horocycle section \textit{rationally linear} if there are rationals $\alpha, \beta \in \mathbb{Q}$ such that the Lebesgue measure of the set $\{t \in \mathbb{R}: \Lambda(t) = \alpha t + \beta \}$ is non-zero.

In \cite{elkies2004gaps}, Elkies and McMullen prove the following equidistribution result, which generalises the equidistribution of long closed horocycles in the modular surface (\cite{sarnak1981asymptotic}).

\begin{theorem}[\cite{elkies2004gaps},Theorem 2.2]\label{non-effective}
Let $n(t)$ be a smooth horocycle section of period $P$ which is not rationally linear. Then, the loops $\{n(t)a(y) \}_{t \in [0,P]}$  equidistribute in the space $X$ as $y \to 0^{+}$. Specifically, for any continuous bounded function $f:X \to \mathbb{R}$ we have that
$$
\frac{1}{P} \int_0^P f(n(t)a(y)) \; dt \to \int_X f \; d \nu.
$$
as $y \to 0^{+}$.
\end{theorem}

Subsequently, many authors have worked to prove effective versions of this theorem of the form
\begin{equation}\label{effectivehorocycle}
    \int_{\mathbb{R}} f(n(t)a(y)) \psi(t) \; dt = \int_X f \; d \nu \int_{\mathbb{R}} \psi(t) \; dt + O(y^{\delta})
\end{equation}
for different horocycle sections $n(t)$, where $f \in C_c^{\infty}(X)$ and $\psi \in C_c^{\infty}(\mathbb{R})$ are smooth and compactly supported. Here the implicit constant will depend upon Sobolev norms $||\cdot||_{C^d_b(X)}$ and $||\cdot ||_{W^{d,p}}$ of the functions $f$ and $\psi$ which are defined precisely in \S \ref{Sobolevnorms}. Indeed,
\begin{itemize}
    \item In \cite{strombergsson2015effective}, Str\"{o}mbergsson proves an effective version of this result when $\xi(t) = (\alpha, \beta)$ and $(\alpha,\beta)$ satisfy certain Diophantine conditions. 
    \item In \cite{browning2016effective}, Browning and Vinogradov prove an effective equidistribution result when $\xi(t) = (t/2,-t^2/4)$, the case which was found to be related to gap distribution of $\sqrt{n}$ by Elkies and McMullen in \cite{elkies2004gaps}. 
    \item Finally, in \cite{vinogradov2016effective}, Vinogradov extends this to horocycle sections for which $\Lambda(t)$ is twice continuous differentiable and sufficiently \textit{nice}. 
\end{itemize}

Now, in \cite{pattison2021rational}, the author extends (\ref{non-effective}) in a different direction by proving the equidistribution of rational points on horocycle sections.

\begin{theorem}\label{original}
Let $n(t)$ be a horocycle section of period $P$ which isn't rationally linear. Then, for any $f \in C_c^{\infty}(X)$ we have that
$$\frac{1}{PN}\sum_{k=0}^{PN-1} f(\Gamma n(k/N) a(y))  \to \int_X f \; d \nu$$
as $N \to \infty$ with $y \asymp \frac{1}{N}$.
\end{theorem}

A natural question to ask is whether, like Theorem \ref{non-effective}, it is possible to make this result effective. Here we answer this question in the affirmative in the case when an equidistribution statement equivalent to (\ref{effectivehorocycle}) holds for $n(t)$.

\begin{theorem}\label{main}
Let $n(t) = (I,\xi(t))u(t)$ be a horocycle section such where $\xi_1$ is continuous and $\Lambda_{\xi}$ is continuously differentiable. Suppose that there exists $\delta > 0$, $d > 0$ such that for all $ \alpha < \beta $ and all $f \in C_b^d(X)$ there exists a constant $C = C(\xi|_{[\alpha,\beta]})$ such that
\begin{equation}\label{initialassumption}
    \bigg|\frac{1}{\beta - \alpha}\int_{\alpha}^{\beta} f(n(t)a(y)) \; dt - \int_X f \; d \nu \bigg| \leq \frac{C L y^{\delta}||f||_{C^d_b(X)}}{\beta - \alpha}
\end{equation}
as $y \to 0^{+}$, where $L := \max(|\alpha|,|\beta|,1)$.
Then, there exists constants $\delta_1, \delta_2 > 0 $ and constants $C_1 = C_1(\xi,\psi)$, $C_2 = C_2(\xi,\psi)$ such that for any compactly supported $C^1$ function $\psi$ and any $f \in C_b^d(X)$, we have that 
\begin{equation}\label{non-primitive case}
\bigg| \frac{1}{N} \sum_{k \in \mathbb{Z}} f(n(k/N)a(1/N)) \psi(k/N) - \int_X f \; d \nu \int_{\mathbb{R}} \psi \; dt \bigg| \leq C_1 ||f||_{C^{d'}_b(X)} N^{-\delta_1}
\end{equation}
as $N \to \infty$ and
\begin{equation}\label{primitive case}
\bigg| \frac{1}{\phi(q)} \sum_{\substack{p \in \mathbb{Z}\\(p,q)=1}} f(n(p/q)a(1/q)) \psi(p/q) - \int_X f \; d \nu \int_{\mathbb{R}} \psi \; dt \bigg| \leq C_2 ||f||_{C^{d'}_b(X)} q^{-\delta_2}
\end{equation}
as $q \to \infty$, where $d' := \max(4,d)$. 
\end{theorem}

\begin{remark}
\textup{As is suggested in this theorem, the constants $C_1$ and $C_2$ should only depend upon the restriction of $\xi$ to the intervals $[\alpha,\beta]$, which in our case will be the smallest closed interval containing the support of the test functions $\psi$. Indeed, we are careful to ensure that this is the case (e.g at the start of the proof of Proposition \ref{non-primitive interval}). To avoid repeated use of cumbersome notation, we will assume this is implicitly understood and write, for example, $C = C(\xi)$ instead of $C = C(\xi|_{[\alpha,\beta]})$ in our arguments.
}
\end{remark}

As a consequence of Theorem \ref{main}, we have the following concrete effective equidistribution results in some of the cases listed above. 

For the case of \cite{strombergsson2015effective}, we say that $\xi = (\xi_1,\xi_2) \in \mathbb{R}^2$ is of Diophantine-type $K \geq 3/2$ with constant $c > 0$ if $|| \xi - q^{-1}m|| > cq^{-K} $ for all $q \in \mathbb{N}$ and $m \in \mathbb{Z}^2$. In the case of a horocycle section with $\xi$ being a fixed element of $\mathbb{R}^2$ with Diophantine type $K$, (\ref{effectivehorocycle}) holds for $\delta = \min(1/4,1/K).$ This gives us the following result.

\begin{corollary}\label{Strom Cor}
Let $n(t) := (I,\xi)u(t)$ where $\xi $ is of finite Diophantine type $K$ with constant $c>0$. Then, there exists $\delta_1 = \delta_1(K), \delta_2 = \delta_2(K) > 0$ and a constant $C = C(c,K) > 0$ such that for all $f \in C^8_b(X)$ and all compactly supported $C^1$ functions $\psi:\mathbb{R} \to \mathbb{R}$ we have that 
\begin{align}
& \bigg| \frac{1}{N} \sum_{k \in \mathbb{Z}} f(n(k/N)a(1/N)) \psi(k/N) - \int_X f \; d \nu \int_{\mathbb{R}} \psi \; dt \bigg| \nonumber \\
\leq & C (W+1)(L^{1/2} + |\xi_1| + |\xi_2| + 1)  ||f||_{C^8_b(X)} ||\psi||_{W^{1,\infty}} N^{-\delta_1} \label{non-prim error}
\end{align}
and 
\begin{align}
    & \bigg| \frac{1}{\phi(q)} \sum_{
    \substack{p \in \mathbb{Z} : \\ (p,q)=1}
    } f(n(p/q)a(1/q)) \psi(p/q) - \int_X f \; d \nu \int_{\mathbb{R}} \psi \; dt \bigg| \nonumber \\ \leq &
    C (W+1)(L^{1/2} + |\xi_1| + |\xi_2| + 1)^{1/2}  ||f||_{C^8_b(X)} ||\psi||_{W^{1,\infty}}q^{-\delta_2} \label{prim error}
\end{align}
where $L$ is the smallest real number $\geq 1$ such that \textup{supp}$(\psi) \subset [-L,L]$ and $W$ is the width of 
\textup{supp}$(\psi)$.
\end{corollary}
Note that this holds for almost every $\xi \in \mathbb{R}^2$ since the set of all $\xi$ which do not have finite Diophantine type has Hausdorff dimension $0$. 

\begin{remark}
\textup{
As stated, the error term in the primitive case of Corollary \ref{Strom Cor} is asymptotically better in $L$. However, in the non-primitive case, the error term (\ref{non-prim error}) can be replaced by (\ref{prim error}) (with $N$ in place of $q$) but this comes at the cost of having a much smaller decay constant $\delta_2$ (see Lemma \ref{c independence}).
}
\end{remark}

In the case of a periodic horocycle sections, it is easy to see that the periodicity allows us to remove any dependence on $L$ in the error term. Indeed we have the following in the setting of \cite{browning2016effective}.

\begin{corollary}\label{Brown Cor}
Let $n(t):= (I,(t/2,-t^2/4))u(t).$ There exists constants $\delta_1, \delta_2 > 0$ and a constant $C > 0$ such that for all functions $f \in C^8_b(X)$ and all compactly supported $C^1$ functions $\psi:\mathbb{R} \to \mathbb{R}$ we have 
\begin{align*}
    & \bigg| \frac{1}{N} \sum_{k \in \mathbb{Z}} f(n(k/N)a(1/N)) \psi(k/N) - \int_X f \; d \nu \int_{\mathbb{R}} \psi \; dt \bigg| \leq & C (W+1) ||f||_{C^8_b(X)} || \psi ||_{W^{1,\infty}} N^{- \delta_1}
\end{align*}
and 
\begin{align*}
    & \bigg| \frac{1}{\phi(q)} \sum_{\substack{p \in \mathbb{Z} : \\ (p,q)=1}} f(n(p/q)a(1/q)) \psi(p/q) - \int_X f \; d \nu \int_{\mathbb{R}} \psi \; dt \bigg|  \leq & C (W+1)||f||_{C^8_b(X)} || \psi ||_{W^{1,\infty}} q^{- \delta_2}
\end{align*}
where $W$ is the width of supp$(\psi)$.
\end{corollary}

The proof of these Corollaries are found in §\ref{concreteEquidistribution}.

\begin{remark}
\textup{
When applied to functions $f \in C^8_b(X)$, which are invariant under the right action of the group $\{I\} \times \mathbb{R}^2 \leq G$, we recover an effective equidistribution result for the points 
$$
\bigg\{
\bigg( \begin{pmatrix}
1 & p/q \\
0 & 1
\end{pmatrix}
\begin{pmatrix}
q^{-1/2} & 0 \\
0 & q^{1/2}
\end{pmatrix}, \frac{p}{q} \bigg)
: (p,q) = 1 \bigg\} 
$$
in SL$(2,\mathbb{Z}) \backslash$ SL$(2,\mathbb{R}) \times \mathbb{R} / \mathbb{Z}$. Such a result was first shown in \cite{einsiedler2021primitive} using $S$-arithmetic extensions of the space SL$(2,\mathbb{Z}) \backslash$ SL$(2,\mathbb{R}) \times \mathbb{R} / \mathbb{Z}$. Our argument avoids the need to use such spaces and is hence shorter, although is it therefore less general. 
More recently Jana, generalizing \cite{burrin2021translates}[Theorem 1.1], has shown an effective equidistribution result for such points with the optimal equidistribution rate of $q^{1/2-\theta - \epsilon}$, where $\theta = 7/64$ is the current best known bound towards the Ramanujan conjecture (\cite{jana2021joint}). This uses spectral theory on the space SL$(2,\mathbb{Z}) \backslash$ SL$(2,\mathbb{R})$. It would be interesting to see if the rate of equidistribution in Corollaries \ref{Strom Cor} and \ref{Brown Cor} could be improved by using such techniques. 
}
\end{remark}

\section{Preliminaries: Sobolev norms}\label{Sobolevnorms}
For $1 \leq p \leq \infty$, $k \in \mathbb{N}_0$ and a $C^d$ function $\psi : \mathbb{R} \to \mathbb{R}$ we define its $L^p$ Sobolev norm of order $d$ by
$$
||\psi||_{W^{d,p}} := \sum_{j = 0}^d ||\psi^{(j)}||_p \in [0,\infty].
$$
As is well known from the theory of Sobolev spaces, this norm can be extended to a much larger class of functions. However, for the purposes of this paper, we merely remark that it is also defined for functions $\psi$ which are piecewise $C^d$ as such functions derivatives are defined almost everywhere.

Let $\mathfrak{g} := \mathfrak{sl}(2,\mathbb{R}) \oplus \mathbb{R}^2 $ be the Lie algebra of $G$. Every element $A \in \mathfrak{g}$ defines a 
differential operator on $C^{\infty}(X)$ via the equation
$$
Af(g) := \frac{d}{dt} f(g \exp(At)) \big\vert_{t=0}.
$$
Now, we choose the basis $\{X_i \}_{i=1}^5$ of $\mathfrak{g}$ where
$$
X_1 =  (X, \mathrm{0}), \;
X_2 = (Y, \mathbf{0}), \;
X_3 = (Z, \mathbf{0}), \;
X_4 = (\mathbf{0}, (1,0)), \text{ and } 
X_5 = (\mathbf{0}, (0,1)) \;
$$
and 
$$
X = \begin{pmatrix}
0 & 1 \\
0 & 0
\end{pmatrix},
Y = \begin{pmatrix}
0 & 0 \\
1 & 0
\end{pmatrix}
\text{ and }
Z = \begin{pmatrix}
1 & 0 \\
0 & -1
\end{pmatrix}.
$$
Let $\mathcal{M}_d$ be the set of monomials $\{ X_i\}_{i=1}^5$ of degree at most $d$. For a function $f \in C^{\infty}(X) $, we define its $(L^{\infty})$ Sobolev norm of degree $d$ by 
$$
||f||_{C_b^d} := \sum_{D \in \mathcal{M}_d} ||Df||_{\infty}.
$$
We let $C_b^d(X)$ be the space of all $d$ times differentiable functions $f:X \to \mathbb{R}$ whose $|| \cdot ||_{C_b^d}$-norm is finite. 

These Sobolev norms have the following properties, discussed in greater generality in \cite{einsiedler2021primitive}[§3]:
\begin{itemize}
    \item (Nesting) For any $d \in \mathbb{N}$ and $f \in C_b^d(X)$ we have that $f \in C_b^{d'}(X)$ for any $d' \leq d$ and 
    \begin{equation}\label{Nesting}
        || f ||_{C_b^{d'}} \leq || f ||_{C_b^{d}}
    \end{equation}
    \item (Product Rule) For any $d \in \mathbb{N}$ and $f_1,f_2 \in C_b^d(X)$ we have that $f_1,f_2 \in C_b^d(X)$ with 
    \begin{equation}\label{ProductRule}
        ||f_1 f_2 ||_{C^d_b} \ll_d  ||f_1 ||_{C^d_b} ||f_2 ||_{C^d_b}
    \end{equation}
    \item (Distance bound) For all $d \geq 1$, $f \in C_b^d(X)$ and $g \in G$ we have that 
    \begin{equation}\label{DistanceBound}
         ||g \cdot f - f||_{\infty} \ll \rho(g,e) ||f||_{C^1_b(X)}.
    \end{equation}
    where $(g \cdot f)(x) := f(xg)$ and $e \in G$ is the identity.
    \item (Adjoint bound) For all $d \in \mathbb{N}$ , $f \in C_b^d(X)$ and $g \in G$ we have that 
    \begin{equation}\label{AdjointBound}
         ||g \cdot f ||_{C^d_b(X)} \ll_d |||g|||^{2d} ||f||_{C^d_b(X)}
    \end{equation}
    where $|||g|||$ is the maximal value of the entries of $g$ and $g^{-1}$.
\end{itemize}

\begin{proof}[Proof of properties]
Nesting (\ref{Nesting}) is immediate from the definition of Sobolev norms. 

In terms of the product rule (\ref{ProductRule}), we have that the standard product rule for differentiation tells us that, whenever $A \in \mathfrak{g}$ and $f_1,f_2 \in C_b^d(X)$ 
$$
A(f_1 f_2) = A(f_1)f_2 + f_1 A(f_2).
$$
Therefore, for any $D \in \mathcal{M}_d$, $D(f_1f_2)$ is a sum of terms of the form $D_1(f_1)D_2(f_2)$ where $D_1,D_2$ are monomials whose degrees sum to $d$. Hence, $||D(f_1 f_2)||_{\infty} \ll_d ||f_1||_{C^b_d} ||f_2||_{C^b_d}$ and so, summing over all $D \in \mathcal{M}_d$, the claimed result follows.

For the distance bound (\ref{DistanceBound}), let $g \in G$ and $w:[0,\rho(g,e)] \to G$ be unit speed path of shortest length between $e$ and $g$. By the fundamental theorem of calculus, for any $h \in G$ and $f \in C_b^d(X)$, 
\begin{equation}\label{FTC}
   (g \cdot f)(h) - f(h) = \int_0^{\rho(g,e)} \frac{d}{ds} f(hw(s)) |_{s=t} \; dt. 
\end{equation}
Now, $\frac{d}{dt} f(hw(s)) |_{s=t} = [Af](hw(t))$ for $A = A(t) = \frac{d}{ds} w(t)^{-1} w(t+s) |_{s=0} \in \mathfrak{g}$. Since the path $w$ is of unit speed, $\rho(w(t)^{-1}w(s+t),e) = s$. So, $A(t)$'s entries are bounded and hence it is a linear combination $\sum_{i=1}^5 \alpha_i X_i $ of the  $X_i$ with coefficients $\alpha_i$ bounded uniformly in $t$ and $g$. Thus, 
$$
\bigg|\frac{d}{ds} f(hw(s)) |_{s=t} \bigg| \ll ||f||_{C_b^1}
$$
for all $t \in [0,\rho(g,e)]$. The result follows from plugging this absolute value bound into (\ref{FTC}). 

Finally, for the adjoint bound (\ref{AdjointBound}), note that whenever $A \in \mathcal{g}$, $f \in C_b^1(X)$ and $g,h \in G$ we have that 
\begin{align*}
    A(g\cdot f)(h) & = \frac{d}{dt} f(h \exp(tA)g) \\
    & = \frac{d}{dt} f(hg \exp(tg^{-1}Ag)) \\
    & = (g \cdot [g^{-1} A g]f )(h).
\end{align*}
For $g \in \{ X_i \}_{i=1}^5$, the size of the entries $g^{-1}Ag$ will be bounded, up to a constant, by $|||g|||^2$. So, when $g^{-1}Ag$ is written as a linear combination of the $X_i$, the coefficients will be bounded, up to a constant, by $|||g|||^2$. More generally, when $D = A_1 \cdot \dotsc \cdot A_d \in \mathcal{M}_d$, we have that
$$
D(g \cdot f) = g \cdot [ (g^{-1}A_1g) \cdot \dotsc \cdot (g^{-1} A_d g) ]f 
$$
and $g^{-1} A_1 g \cdot \dotsc \cdot g^{-1} A_d g$ can be written as a sum of monomials in $\mathcal{M}_d$ whose coefficients are bounded, up to a constant, by $(|||g|||^2)^d = |||g|||^{2d}. $

\end{proof}

\section{The non-primitive case}

In \cite{pattison2021rational}, Theorem \ref{original} was shown using the mixing properties of the flow defined by the right action of the one-parameter subgroup $u(t)$, together with \cite{elkies2004gaps}[Theorem 2.2].
To extend this result to an effective version, we show that this flow has a polynomial rate of mixing, as is outlined by Str\"{o}mbergsson in \cite{strombergsson2015effective}[Remark 11.1].

\begin{proposition}\label{MixingRate}
For any $f,h \in C_b^{4}(X)$ we have that
$$
\int_X (u(t) \cdot f) \; h \; d\nu = \int_X f \; d\nu \int_X h \; d \nu + O_{\epsilon}((1+t)^{-1+ \epsilon} ||f||_{C_b^4}||h||_{C_b^4})
$$
\end{proposition}

\begin{proof}
Let $f,h \in C_b^4(X)$ and $t \in \mathbb{R}$.
For $\theta \in \mathbb{R}$  we set $$
k(\theta) := \bigg(\begin{pmatrix}
\cos(\theta) & -\sin(\theta) \\
\sin(\theta) & \cos(\theta) 
\end{pmatrix}, (0,0) \bigg).
$$
As in \cite{ratner1987rate}[Lemma 2.3], we can find $\theta_1 = \theta_1(t), \theta_2 = \theta_2(t) \in \mathbb{R}$ such that $u(t) = k(\theta_1) \Phi(2 \sinh^{-1}(t)) k(\theta_2)$. Now, the flow given by the right action of the one-parameter subgroup $\Phi(t)$ is exponentially mixing. Indeed, by \cite{edwards2013rate} we know that 
$$
\int_X f_1(x) f_2(x \Phi(s)) \; d \nu(x) = \int_X f_1 \; d \nu \int_X f_2 \; d \nu + O(se^{-s/2} ||f_1||_{C^4_b} ||f_2||_{C^4_b} )
$$
for any $f_1, f_2 \in C_b^d(X)$. Setting $f_1 = k(-\theta_1)\cdot h$, $f_2 = k(\theta_2)\cdot f$, $s = 2 \sinh^{-1}(t)$ and using the right invariance of the measure $\nu$ we obtain that 
\begin{align*}
    \int_X  f(x k(\theta_1) \Phi(2 \sinh^{-1}(t)) k(\theta_2)) h(x) \; d \nu(x) =  \int_X f \; d \nu \int_X h \; d \nu \\ +
     O(\sinh^{-1}(t) e^{-\sinh^{-1}(t)} ||k(-\theta_1) \cdot h||_{C^4_b} ||k(\theta_2) \cdot f||_{C^4_b})
\end{align*}
Note that, by the compactness of the group $K:= \{ k(\theta): \theta \in \mathbb{R}\} \leq G$, the implicit constant in the bound (\ref{AdjointBound}) for elements $g \in K$ can be made independent of the choice of $g$. Using this, the fact that $\sinh^{-1}(t) \sim \ln{t} $ as $t \to \infty$ and that $u(t) = k(\theta_2) \Phi(2 \sinh^{-1}(t)) k(\theta_1)$ we get the desired error term.
\end{proof}

To prove Theorem \ref{main} we first show the result in the non-primitive case. Following \cite{einsiedler2021primitive}, for a function $f \in C_b^d(X)$ and $M \in \mathbb{N}$, we define the discrepancy 
\begin{equation}
    D_Mf(x) := \frac{1}{M} \sum_{m=0}^{M-1} f(xu(m)) - \int_X f \; d \nu.
\end{equation}

This is relevant in our setting as moving from $n(k/N)a(1/N)$ to $n((k+m)/N)a(1/N)$ corresponds roughly to right multiplying by $u(m)$, as the next lemma shows.

\begin{lemma}\label{distancedifference}
There exists a constant $c > 0$ such that the following holds.
Let $n(t)= (1,\xi(t))u(t)$ be a horocycle section with $\xi_1$ continuous and $\Lambda_{\xi}$ Lipschitz. Let $\alpha < \beta$, $t \in \mathbb{R}$ and $N,m > 0$ be real numbers with $t,t+m/N \in [\alpha,\beta]$. Let $A(\xi) := ||\xi_1||_{L^{\infty}[\alpha,\beta]} + L_{\xi}$ where
$L_{\xi}$ is a Lipschitz constant for $\Lambda_{\xi}|_{[\alpha,\beta]}$. Then 
\begin{equation}\label{distancedifference1}
    \rho(n(t)a(1/N)u(m), n(t+m/N)a(1/N)) \ll (1 + m)N^{-1/2} A(\xi)
\end{equation}
provided that $(1 + m)N^{-1/2} A(\xi) \leq c$.

Similarly, if $\alpha < \beta$, $s,t \in [\alpha,\beta]$ and $N > 0$ then we have that
\begin{equation}\label{distancedifference2}
    \rho(n(s)a(1/N),n(t)a(1/N)) \ll N|s-t| + (1 + |s-t|)A(\xi)N^{-1/2} ,
\end{equation}
provided that $N|s-t| + (1 + |s-t|)A(\xi)N^{-1/2} \leq c$.
\begin{proof}
For the first bound, writing 
\begin{align*}
    n(t+m/N)a(1/N) &= (I,\xi(t+m/N))u(t+m/N)a(1/N) \\
    n(t)a(1/N)u(m) &= (I,\xi(t))u(t+m/N)a(1/N)
\end{align*}
 and using left invariance, we have that
\begin{align*}
    & \rho(n(t)a(1/N)u(m), n(t+m/N)a(1/N)) \\
    = & \rho((I,\xi(t+m/N)-\xi(t)) n(t+m/N)a(1/N)), n(t+m/N)a(1/N)) \\
    = & \rho((I, (\xi(t+m/N)-\xi(t))n(t+m/N)a(1/N)) ,(I,(0,0))).
\end{align*}
We then calculate that 
\begin{align*}
    & (\xi(t+m/N)-\xi(t))n(t+m/N)a(1/N) \\= & ( (\xi_1(t+m/N)-\xi_1(t))N^{-1/2} , \\ &( (t+m/N)(\xi_1(t+m/N)- \xi_1(t)) + \xi_2(t+m/N) - \xi_2(t) )N^{1/2} ).
\end{align*}
The first component has absolute value at most $2 ||\xi_1||_{L^{\infty}[\alpha, \beta]} N^{-1/2}$ and the second, using the Lipschitz condition on $\xi_2$, has absolute value bounded above by $(2L_{\xi} + ||\xi_1||_{L^{\infty}[\alpha, \beta]})mN^{-1/2}$. 
Using that there is a neighbourhood of $(I,(0,0)) \in G$ around which $\rho$ is Lipschitz equivalent to metric on $G$ induced from, say, the $L^{\infty}$-norm on $\mathbb{R}^6 \cong M_{2 \times 2}(\mathbb{R}) \times \mathbb{R}^2$ (\cite{einsiedler2013ergodic}[Lemma 9.12]), we can find a constant $c > 0$ such that, provided that $$
||\xi_1||_{L^{\infty}[\alpha, \beta]} N^{-1/2} + (2L_{\xi} + ||\xi_1||_{L^{\infty}[\alpha, \beta]})mN^{-1/2} \leq 2 A(\xi)(1+m)N^{-1/2} \leq 2c
$$
we have 
$$
\rho(n(t)a(1/N)u(m), n(t+m/N)a(1/N)) \ll A(\xi)(1 + m)N^{-1/2}.
$$

The proof of the second bound is similar to the first. However, now we calculate that
\begin{align*}
    & \rho(n(s)a(1/N),n(t)a(1/N)) \\
    = & \rho(u(N(s-t)), (\xi(s)-\xi(t)) n(s)a(1/N)),(I,(0,0)).
\end{align*}
So we find that $$
|| (\xi(s)-\xi(t)) n(s)a(1/N)) ||_{\infty} \ll N^{-1/2} (1 + |s-t|)(||\xi_1||_{L^{\infty}[\alpha,\beta]} + L_{\xi}).
$$
Together with the fact that, for any matrix norm $|| \cdot ||$ we have
$
|| n(N(s-t))- I || \ll N|s-t|,
$
the result follows as above.
\end{proof}
\end{lemma}

In the following proposition, which is adapted from \cite{einsiedler2021primitive}[Proposition 5.3], we use the notation $[a,b]_{\mathbb{Z}}$ to denote all integers in the interval $[a,b]$. 

\begin{proposition}\label{non-primitive interval}
Suppose that $n(t) = (I,\xi(t))u(t)$ satisfies the conditions of Theorem \ref{main} for some $\delta > 0$, $d \in \mathbb{N}$ and $C = C(\xi)$. Then, there exists $\delta' > 0$ such that for any $\alpha < \beta$ with $\beta - \alpha \leq 1$ and $f \in C_b^{d'}(X)$ we have that 
\begin{equation}\label{rationalequi1}
    \bigg| \frac{1}{|[N \alpha, N \beta]_{\mathbb{Z}}|} \sum_{k \in [N \alpha, N \beta]_{\mathbb{Z}}} f(n(k/N)a(1/N)) - \int f \; d \nu \bigg| \ll \frac{C'(\xi,L)||f||_{C^{d'}_b(X)}}{(\beta - \alpha)N^{\delta'}}
\end{equation}
where $C'(\xi,L) = C(\xi)^{1/2}L^{1/2} + A(\xi) + 1$ with $A(\xi)$ is defined as in Lemma \ref{distancedifference}, $L = \max(|\alpha|, |\beta|,1)$ and $d' = \max(d,4)$

\begin{proof}
In the following argument, we will show the result where $A(\xi)$ is replaced by $A'(\xi)$, where $A'$ is defined as $A$ is in Lemma \ref{distancedifference}, but with respect to the interval $[\alpha, \beta + (\beta - \alpha)]$ or, by symmetry under negating the values of $m$ and $M$ in the following proof, $[\alpha - (\beta - \alpha), \beta]$. The result will then follow with $A$ in place of $A'$ since we first can split the interval $[ \alpha,  \beta]$ into $[ \alpha, (\alpha + \beta)/2] $ and $[(\alpha + \beta)/2, \beta] $ and apply the result with $A'$ in the error term to these respective sections.

Let $\kappa \in (0,1/4)$, $M \in [1,N^{1/4}]_{\mathbb{Z}}$ be a constants to be later specified and define $I := (-1/2N^{\kappa}, 1/2N^{\kappa})$. Throughout the proof, we will make the assumption that $N$ is sufficiently large so that $N^{-\kappa}A'(\xi) \leq c$. This will allow us to apply Lemma \ref{distancedifference} in the following arguments. Note that, in the case when $N^{-\kappa}A'(\xi) \leq c$, provided that $\delta' \leq \kappa$, we have that (\ref{rationalequi1}) (with $A'$ in place of $A$) holds trivially for an implicit constant independent of any other parameters, so this restriction is made without loss of generality. Similarly, we will assume that $N^{3/4} (\beta - \alpha) \geq 1$ so that $|[N \alpha, N \beta]_{\mathbb{Z}}| \asymp N(\beta - \alpha)$ and so, again, we can apply Lemma \ref{distancedifference} in the following argument.

Now, for any $k,N \in \mathbb{N}$ and $t \in k/N + (1/N)I$, (\ref{distancedifference2}) in Lemma \ref{distancedifference} and (\ref{DistanceBound}) tell us that $|f(n(k/N)a(1/N)) - f(n(t)a(1/N))| \ll A'(\xi) ||f||_{C^d_b(X)} N^{-\kappa}$. 
We therefore have that 
\begin{align}
    & \frac{1}{|[N \alpha, N \beta]_{\mathbb{Z}}|} \sum_{k \in [N \alpha, N \beta]_{\mathbb{Z}}} f(n(k/N)a(1/N)) \nonumber \\
    = & \frac{N^{1 + \kappa}}{|[N \alpha, N \beta]_{\mathbb{Z}}|} \sum_{k \in [N \alpha, N \beta]_{\mathbb{Z}}} \int_{k/N + (1/N)I} f(n(t)a(1/N)) \; dt + O( A'(\xi)||f||_{C^1_b(X)} N^{-\kappa}). \label{error1}
\end{align} 
Now, we have that
\begin{align}
    & \frac{N^{1 + \kappa}}{|[N \alpha, N \beta]_{\mathbb{Z}}|} \sum_{k \in [N \alpha, N \beta]_{\mathbb{Z}}} \int_{k/N + (1/N)I} f(n(t)a(1/N)) \; dt - \int_X f \; d \nu  \nonumber \\
    = & \frac{N^{1 + \kappa}}{|[N \alpha, N \beta]_{\mathbb{Z}}|} \sum_{k \in [N \alpha, N \beta]_{\mathbb{Z}}} \int_{k/N + (1/N)I} \frac{1}{M} \sum_{m=0}^{M-1} f(n(t + m/N)a(1/N)) \; dt - \int_X f \; d \nu \nonumber \\
    + & O\bigg( \frac{M||f||_{\infty}}{|[N \alpha, N \beta]_{\mathbb{Z}}|} \bigg), \label{error2}
\end{align}
since 
$$
 \sum_{k \in [N \alpha, N \beta]_{\mathbb{Z}}} N^{1 + \kappa} \int_{k/N +  (1/N)I} f(n(t)a(1/N)) \; dt 
$$
and
$$
 \sum_{k \in [N \alpha, N \beta]_{\mathbb{Z}}} N^{1 + \kappa} \int_{k/N + (1/N)I} f(n(t + m/N)a(1/N)) \; dt 
$$
differ by $m$ terms of size  at most $ ||f||_{\infty}N^{-1-\kappa}$ and we are summing $M$ such terms. By (\ref{distancedifference1}) in Lemma \ref{distancedifference} and (\ref{DistanceBound}), we find that $$|f(n(t + m/N)a(1/N)) - f(n(t)a(1/N)u(m))| \ll A'(\xi)||f||_{C_b^1(X)} MN^{-1/2}.$$ Here the fact that $N^{3/4}|\beta - \alpha| \geq 1$ and $m \leq M \leq N^{1/4}$ are used to allow us to apply Lemma \ref{distancedifference}. So
\begin{align}
    & \frac{N^{1 + \kappa}}{|[N \alpha, N \beta]_{\mathbb{Z}}|} \sum_{k \in [N \alpha, N \beta]_{\mathbb{Z}}} \int_{k/N + (1/N)I} f(n(t)a(1/N)) \; dt - \int_X f \; d \nu \nonumber \\
    = & \frac{N^{1 + \kappa}}{|[N \alpha, N \beta]_{\mathbb{Z}}|} \int_J [D_Mf](n(t)a(1/N)) \; dt \nonumber \\
    + & O\bigg((A'(\xi)||f||_{C_b^1(X)} \bigg(\frac{M}{N^{1/2}} + \frac{M}{|[N \alpha, N \beta]_{\mathbb{Z}}|} \bigg) \label{error3} \bigg)
\end{align}
where $J := \cup_{k \in [N \alpha, N \beta]_{\mathbb{Z}}}(k/N + (1/N) I)$.

Now, using the Cauchy-Schwartz inequality we have that 
$$
\bigg| \int_J [D_Mf](n(t)a(1/N)) \; dt  \bigg| \leq \text{Vol}(J)^{1/2} \bigg( \int_{\alpha - 1/2N^{\kappa}}^{\beta + 1/2N^{\kappa}} |D_Mf|^2(n(t)a(1/N)) \; dt \bigg)^{1/2}. 
$$
The key point now is that we can apply (\ref{initialassumption}). Indeed, we have that
\begin{align*}
    & \int_{\alpha - 1/2N^{\kappa}}^{\beta + 1/2N^{\kappa}} |D_Mf|^2(n(t)a(1/N)) \; dt = (\beta - \alpha + N^{-1-\kappa})\int_X |D_Mf|^2 \; d \nu \\
    + & O(C(\xi)L ||[D_Mf]^2||_{C^d_b(X)}N^{- \delta}),
\end{align*}
where $L = \max(|\alpha|, |\beta|,1)$. We now seek to bound the two terms above using mixing and the properties of Sobolev norms. Starting with the integral, one finds by expanding that 
\begin{align}
    \int_X |D_Mf|^2\; d \nu = \frac{1}{M^2} \sum_{m_1=0}^{M-1}\sum_{m_2=0}^{M-1} \langle u(m_1) \cdot f_0 ,u(m_2) \cdot f_0 \rangle \label{doublesum}
\end{align}
where $f_0 := f - \int_X f \; d \nu$ and $\langle.,.\rangle$ is the inner product on $L^2(X,\nu)$. By Proposition \ref{MixingRate} and the right-invariance of $\nu$, one sees that
\begin{align*}
    \langle u(m_1) \cdot f_0 ,u(m_2) \cdot f_0 \rangle & \ll_{\epsilon} [|m_1-m_2|+1]^{-1+\epsilon} || f_0 ||_{C^4_b(X)}^2  \\
    & \ll_{\epsilon} [|m_1-m_2|+1]^{-1+\epsilon} || f ||_{C^4_b(X)}^2.
\end{align*}

Therefore, since the number of terms in the sum (\ref{doublesum}) where $|m_1-m_2| = k$ is at most $k$, if $k \in \{0,1,2,\dotsc M-1\}$, and is 0 otherwise, we have that
\begin{equation}\label{error4}
    \int_X |D_Mf|^2\; d \nu \ll_{\epsilon} \frac{1}{M} \sum_{k=0}^{M-1} (1+k)^{-1+\epsilon} || f ||_{C^4_b(X)}^2 \ll_{\epsilon} M^{-1+\epsilon} || f ||_{C^4_b(X)}^2.
\end{equation}

For the estimate of the Sobolev norm, we use (\ref{AdjointBound}) to see that
\begin{align}
    |||D_Mf|^2||_{C^d_b(X)} & \leq  \frac{1}{M^2} \sum_{m_1=0}^{M-1}\sum_{m_2=0}^{M-1} || (u(m_1) \cdot f_0)(u(m_2) \cdot f_0) ||_{C^d_b(X)} \nonumber \\
    & \ll \frac{1}{M^2}\bigg( \sum_{m_1=0}^{M-1} || u(m_1)\cdot f_0 ||_{C^d_b(X)} \bigg) \bigg( \sum_{m_2=0}^{M-1} || u(m_2)\cdot f_0 ||_{C^d_b(X)} \bigg) \nonumber \\
    & \ll \frac{1}{M^2}\bigg( \sum_{m_1=0}^{M-1} (1+m_1)^{2d} ||f_0 ||_{C^d_b(X)} \bigg) \bigg( \sum_{m_2=0}^{M-1} (1+m_2)^{2d} || f_0 ||_{C^d_b(X)} \bigg) \nonumber \\
    & \ll M^{2d} ||f||_{C^d_b(X)}^2 \label{error5}
\end{align}
Putting (\ref{error1}), (\ref{error2}), (\ref{error3}), (\ref{error4}) and (\ref{error5}) together, and using that Vol$(J) \asymp (\beta - \alpha)N^{- \kappa}$ we get that, for $d' = \max(4,d)$
\begin{align*}
    & \bigg| \frac{1}{|[N \alpha, N \beta]_{\mathbb{Z}}} \sum_{k \in [N \alpha, N \beta]_{\mathbb{Z}}} - \int_X f \; d \nu \bigg| \\
    \ll_{\epsilon} &  \bigg( A'(\xi)N^{-\kappa} + \frac{M}{N(\beta - \alpha)} + A'(\xi) \bigg(\frac{M}{N^{1/2}} + \frac{M}{N(\beta - \alpha)} 
    \bigg)  \\ +
    & N^{\kappa/2}(\beta - \alpha)^{-1/2} ((\beta - \alpha + N^{-1-\kappa})M^{-1 + \epsilon} + C(\xi)LM^{2d} N^{-\delta} )^{1/2} 
    \bigg) ||f||_{C^{d'}_b}.
\end{align*}
To firstly tidy this error term up, note that $|\alpha|, |\beta| \leq L$, $M \leq N^{1/4}$ and $\kappa \leq 1/4$ and so we can write the error term as 
$$
O_{\epsilon}\bigg((N^{-\kappa} + M^{d} N^{\kappa/2 - \delta/2} + N^{\kappa/2}M^{-1/2 + \epsilon/2} )\frac{(A'(\xi) + C(\xi)^{1/2}L^{1/2} + 1)||f||_{C^{d'}_b(X)} }{\beta - \alpha}\bigg).
$$
One can then see that choosing $M = N^{\delta/(2d'+1)}$ and then $\kappa = \delta/3(2d+1)$ balances these error terms and shows that for any $\delta' < \kappa$
\begin{align*}
        &\bigg| \frac{1}{|[N \alpha, N \beta]_{\mathbb{Z}}|} \sum_{k \in [N \alpha, N \beta]_{\mathbb{Z}}} f(n(k/N)a(1/N)) - \int f \; d \nu \bigg| \\ \ll & \frac{(C(\xi)^{1/2}L^{1/2} + A'(\xi) + 1)||f||_{C^d_b(X)}}{(\beta - \alpha)N^{\delta'}}.
\end{align*}
This completes the proof due to our earlier observations about replacing $A'(\xi)$ with $A(\xi)$. 

\end{proof}
\end{proposition}

We are now in the position to prove (\ref{non-primitive case}) in Theorem \ref{main} via a standard approximation argument. 

\begin{corollary}\label{to compact support}
Suppose that $n(t) = (I,\xi(t))u(t)$ satisfies the conditions of Theorem \ref{main} for some $\delta > 0$, $d \in \mathbb{N}$ and $C = C(\xi)$. Then there exists $\delta_1 > 0$
such that for $d' := \max(d,4)$ any $f \in C_b^{d'}(X)$ and any $C^1$ compactly supported $\psi:\mathbb{R} \to \mathbb{R}$ we have that
\begin{align}
        & \frac{1}{N} \sum_{k \in \mathbb{Z}} f(n(k/N)a(1/N)) \psi(k/N) = \int_X f \; d \nu \int_{\mathbb{R}} \psi(t) dt  \nonumber \\
        + & O((W+1) C'(\xi,L) ||f||_{C^{d'}_b(X)} ||\psi||_{W^{1,\infty}} N^{- \delta_1}) \label{nodiscbound},
\end{align}
where $L$ is the smallest real number $\geq 1$ such that $[-L,L]$ contains \textup{supp}$(\psi)$, $W$ is the width of \textup{supp}$(\psi)$ and, in the constant $C'(\xi,L)$, $A(\xi)$ and $C(\xi)$ are defined with respect to the smallest interval $[\alpha, \beta]$ containing the support of $\psi$. 

\begin{proof}
Let $\kappa > 0$ be a parameter to be later specified.

Denoting $f_0 = f - \int_X f \; d \nu$, we can write 
\begin{align*}
    & \frac{1}{N} \sum_{k \in \mathbb{Z}} f(n(k/N)a(1/N)) \psi(k/N) \\
    = & \frac{1}{N} \sum_{k \in \mathbb{Z}} f_0(n(k/N)a(1/N)) \psi(k/N) + \frac{1}{N} \int_X f \; d \nu \sum_{k \in \mathbb{Z}} \psi(k/N).
\end{align*}
Via approximating $\psi(k/N)$ by integrals of length $1/N$ centred at $k/N$, one can easily see that 
\begin{equation}\label{RiemannSum}
    \frac{1}{N} \sum_{k \in \mathbb{Z}}\psi(k/N) = \int_{\mathbb{R}} \psi(t) \; dt + O\bigg( \frac{(W+1)||\psi'||_{\infty}}{N} \bigg).
\end{equation}
For $r \in \mathbb{Z}$ we let $ I_r^{\kappa} := [r/N^{\kappa} - 1/2N^{\kappa} , r/N^{\kappa} + 1/2N^{\kappa})$
\begin{align}
    & \frac{1}{N} \sum_{k \in \mathbb{Z}} f_0(n(k/N)a(1/N)) \psi(k/N) = \frac{1}{N}\sum_{r \in \mathbb{Z}} \sum_{ \substack{k \in \mathbb{Z} : \\ k/N \in I_r^{\kappa}  }} f_0(n(k/N)a(1/N)) \psi(k/N)  \nonumber\\
    = &\frac{1}{N}\sum_{r \in \mathbb{Z}} \psi(r/N^{\kappa}) \sum_{ \substack{k \in \mathbb{Z} : \\ k/N \in I_r^{\kappa}  }} f_0(n(k/N)a(1/N)) \label{term1} \\
    + & \frac{1}{N}\sum_{r \in \mathbb{Z}} \sum_{ \substack{k \in \mathbb{Z} : \\ k/N \in I_r^{\kappa}  }} f_0(n(k/N)a(1/N)) [ \psi(k/N) - \psi(r/N^{\kappa})] \label{term2}
\end{align}

Now, applying Proposition \ref{non-primitive interval}, we find that, 
\begin{equation}
    \frac{1}{N^{1- \kappa}} \sum_{ \substack{k \in \mathbb{Z} : \\ k/N \in I_r^{\kappa}  }} f_0(n(k/N)a(1/N)) \ll C'(\xi,L) ||f||_{C^{d'}_b(X)} N^{\kappa - \delta'}
\end{equation}
for some $\delta' > 0$. Given that there are at most $N^{\kappa}(W+1)$ terms $r$ in (\ref{term1}) for which $\psi(r/N^{\kappa})$ is non-zero, we have a bound on (\ref{term1}) of $$(W+1)C'(\xi,L) ||f||_{C^{d'}_b(X)} ||\psi||_{\infty} N^{\kappa - \delta'}.$$

Turning to (\ref{term2}), we see that for all $k,r$ in the sum,
$|\psi(k/N) - \psi(k/N^{\kappa})| \leq ||\psi'||_{\infty} N^{-\kappa}$. Thus, given that the number of non-zero terms in the double sum is at most $(W+1)N$, we find that (\ref{term2}) us bounded up to a constant by $ (W+1) ||\psi'||_{\infty} N^{-\kappa} ||f||_{\infty}$. Choosing $\kappa = \delta'/2$ and combining (\ref{RiemannSum}) with the bounds on (\ref{term1}) and (\ref{term2}) then gives the result. 
\end{proof}
\end{corollary}

\begin{remark}\label{Piecewise continuous}
\textup{
Corollary \ref{to compact support} also applies to functions $\psi$ which are piecewise $C^1$ and compactly supported. If such a function has $m$ discontinuities then the bound we obtain in (\ref{nodiscbound}) is 
$(W+1) C'(\xi,L) ||f||_{C^{d'}_b(X)} ||\psi||_{W^{1,\infty}} N^{- \delta_1} + m ||f||_{\infty}||\psi||_{\infty} N^{-1}.
$ This arises due to the fact that there are at most $m$ intervals $I_r^{\kappa}$ containing the discontinuities, on which we apply a crude absolute value bound. 
}
\end{remark}

\section{The primitive case}

Turning now to the primitive case, our approach is going revolve around using Fourier analysis. Indeed, for a compactly supported function $\psi$, we will view the weighted horocycle sections averages $f(n(t)a(y))\psi(t)$ as being embedded into a circle $\mathbb{R}/\mathbb{Z}$. This will allow us to decompose the weighted averaged across primitive rational points into averages across non-primitive rational points which have been twisted by characters $e(ct) : = \exp(2 \pi i ct)$. The key point, which is formalized by the following lemma, is that these twisted averages will decay uniformly in the choice of character when the function $f$ has zero mean. This lemma
inspired by \cite{ubis2017effective}[Lemma 2.3] which uses a similar Van-der-corput style differencing to that used in \cite{venkatesh2010sparse}.  

\begin{lemma}\label{c independence}
Suppose that $n(t) = (I,\xi(t))u(t)$ satisfies the conditions of Theorem \ref{main} for some $\delta > 0$, $d \in \mathbb{N}$ and $C = C(\xi)$. Then there exists $\delta_2 > 0$
such that for $d' := \max(d,4)$ any $f \in C_b^{d'}(X)$ of mean zero, any continuously differentiable compactly supported $\psi:\mathbb{R} \to \mathbb{R}$ and any $c \in \mathbb{R}$ we have that 
\begin{equation}\label{twistedsum}
\frac{1}{N} \sum_{k \in \mathbb{Z}} f(n(k/N)a(1/N)) \psi(k/N)e(ck/N) \ll \frac{(W+1) C'(\xi,L)^{1/2} ||f||_{C^{d'}_b(X)} ||\psi||_{W^{1,\infty}}}{N^{\delta_2}}
\end{equation}
where $L,W$ and $C'(\xi,L)$ are defined as in Corollary \ref{to compact support}.

\begin{proof}
As in the proof of Proposition \ref{non-primitive interval}, our argument involves introducing shifts. So, in the following argument, the result is shown where $A(\xi)$, $C(\xi)$ and $C'(\xi,L)$ are replaced by $A_N'(\xi)$, $C_N(\xi)$ and so $C'_N(\xi,L)$, which are defined with respect to the interval $[\alpha- N^{-1/2},\beta + N^{1/2}]$ as opposed to the smallest interval $[\alpha,\beta]$ containing supp$(\psi)$. The stated result follows by the same argument after first replacing $\psi$ by $\psi$ multiplied by the indicator function of the interval $[\alpha-N^{-1/2},\beta + N^{-1/2}]$. This will introduce an error $O(||f||_{\infty} ||\psi||_{\infty}N^{-1/2})$ resulting from the terms in the sum of (\ref{twistedsum}) where $k/N \in [\alpha-N^{-1/2},\alpha) \cup (\beta,\beta+N^{1-/2}]$.
Note that we can still apply Corollary \ref{non-primitive interval} to such functions, as noted in Remark \ref{Piecewise continuous}. 

Let $F_N(t) := f(n(t)a(1/N))\psi(t)$. 
Define 
$$
I:= \frac{1}{N}\sum_{k \in \mathbb{Z}} F_N(k/N) e(cK/N).
$$
Let $T \in [1,N^{1/2}]_{\mathbb{Z}}$ be a parameter to be later specified. By introducing shifts in the sum by an amount $s \in \{0,1,\dotsc,T-1\}$
we see that
\begin{align*}
    I & = \frac{1}{T} \sum_{s = 0}^{T-1} \bigg( \frac{1}{N} \sum_{k \in \mathbb{Z}} F_N((k+s)/N) e(c(k+s)) \bigg) + O\bigg( \frac{T||f||_{\infty}||\psi||_{\infty}}{N} \bigg) \\
    & = \frac{1}{N}\sum_{k \in \mathbb{Z}} \bigg( \frac{1}{T} \sum_{s = 0}^{T-1} F_N((k+s)/N) e(c(k+s)) \bigg) + O\bigg( \frac{T||f||_{\infty}||\psi||_{\infty}}{N} \bigg).
\end{align*}
Therefore, 
\begin{equation}\label{ItoI'}
    |I| \leq \frac{1}{N}\sum_{k \in \mathbb{Z}} \bigg| \frac{1}{T} \sum_{s = 0}^{T-1} F_N((k+s)/N) e(c(k+s)) \bigg| + O\bigg( \frac{T||f||_{\infty}||\psi||_{\infty}}{N} \bigg).
\end{equation}
Let 
$$
I' := \frac{1}{N}\sum_{k \in \mathbb{Z}} \bigg| \frac{1}{T} \sum_{s = 0}^{T-1} F_N((k+s)/N) e(c(k+s)) \bigg|.
$$
Using the Cauchy-Schwartz inequality, expanding and rearranging we have that 
\begin{align*}
    |I'|^2 & \leq \frac{W+1}{N} \sum_{k \in \mathbb{Z}} \bigg| \frac{1}{T} \sum_{s = 0}^{T-1} F_N((k+s)/N)e(c(k+s)/N) \bigg|^2 \\
    & \leq \frac{W+1}{T^2} \sum_{s_1=0}^{T-1}\sum_{s_1=0}^{T-1} 
    \frac{e(c(s_1-s_2)/N)}{N} \sum_{k \in \mathbb{Z}} F_N((k+s_1)/N) F_N((k+s_2)/N) \\
    & \leq \frac{W+1}{T^2}  \sum_{s_1=0}^{T-1}\sum_{s_1=0}^{T-1} \bigg| \frac{1}{N} \sum_{k \in \mathbb{Z}} F_N((k+s_1)/N) F_N((k+s_2)/N) \bigg| \\
    & \leq \frac{W+1}{T^2}  \sum_{s_1=0}^{T-1}\sum_{s_1=0}^{T-1} \bigg| \frac{1}{N} \sum_{k \in \mathbb{Z}} F_N(k/N) F_N((k+s_2-s_1)/N) \bigg| \\
    & \leq \frac{W+1}{T} \sum_{s \in [-T,T]_{\mathbb{Z}} }\bigg| \frac{1}{N} \sum_{k \in \mathbb{Z}}  F_N(k/N) F_N((k+s)/N) \bigg|
\end{align*}
where in the last sum we have let $s = s_2-s_1$.
Now, note for $s$ in the above sum, we have that $(k+s)/N  \in [\alpha-N^{-1/2},\beta + N^{-1/2}]$, and so by Lemma \ref{distancedifference} and (\ref{DistanceBound})
\begin{align*}
    & f(n((k+s)/N)a(1/N)) = f(n(k/N)a(1/N)u(s)) \\
+ & O(TN^{-1/2}A_N'(\xi)||f||_{C^1_b(X)}). 
\end{align*}
Therefore, using the notation $\psi_a(t) := \psi(t+a)$ we have that $|I'|^2$ is bounded by
\begin{align*}
    \frac{W+1}{T}  \sum_{s \in [-T,T]_{\mathbb{Z}}} \bigg| \frac{1}{N} \sum_{k \in \mathbb{Z}}
& ([u(s) \cdot f][f])(n(k/N)a(1/N))
\psi_{s/N}(k/N) \psi(k/N) \bigg| \\
+ O(TN^{-1/2}A'(\xi)||f||_{C^1_b(X)}).
\end{align*}

By Corollary $\ref{to compact support}$ we see that 
\begin{align}
   & \frac{1}{N} \sum_{k \in \mathbb{Z}}
([u(s) \cdot f][f])(n(k/N)a(1/N)) 
\psi_{s/N}(k/N) \psi(k/N) \nonumber \\
=& \int_X ([u(s) \cdot f]f) \; d \nu \int_{\mathbb{R}} \psi_{s/N}(t) \psi(t) \; dt \label{Interror} \\
+ & O((W+1) C_N'(\xi,L) ||[u(s)\cdot f][f] ||_{C^{d'}_b(X)} ||\psi_{s/N} \cdot \psi||_{W^{1,\infty}} N^{- \delta_1}) \label{RemainingError}.
\end{align}

As in the proof of Proposition \ref{non-primitive interval} we have that
$$
\int_X ([u(s) \cdot f][f]) \; d \nu \ll_{\epsilon} (1+|s|)^{-1+\epsilon}||f||_{C^4_b(X)}
$$
and so we have an overall bound on (\ref{Interror}) of $$ O_{\epsilon}((W+1)(1+|s|)^{-1+\epsilon}||f||^2_{C^4_b(X)} ||\psi||^2_{\infty}) .$$
Similarly, for $s \in [-T,T]$ we have that $||[u(s)\cdot f][f] ||_{C^{d'}_b(X)} \ll (1+|s|)^{2d'} ||f||^2_{C^{d'}_b(X)}$. Applying the product rule then gives us a bound of 
$$O((W+1) C_N'(\xi,L) (1+|s|)^{2d'} ||f||^2_{C^{d'}_b(X)} ||\psi||^2_{W^{1,\infty}}) N^{-\delta_1}  $$
on (\ref{RemainingError}).
This gives us the overall bound 
$$
|I'|^2 \ll_{\epsilon} (W+1)^2C_N'(\xi,L)||f||_{C^{d'}_b(X)}^2 ||\psi||_{W^{1,\infty}}^2 (T^{-1+\epsilon} + TN^{-1/2} + T^{2d'} N^{-\delta_1}).
$$
Choosing $T = N^{\delta_1/1+2d'}$ and combining plugging the resulting bound on $I'$ into (\ref{ItoI'}) then gives the result for any $\delta_2 < \delta_1/(2+4d')$.

\end{proof}

\end{lemma}

\begin{corollary}\label{to primitive compact support}
Suppose that $n(t) = (I,\xi(t))u(t)$ satisfies the conditions of Theorem \ref{main} for some $\delta > 0$, $d \in \mathbb{N}$ and $C = C(\xi)$. Then there exists $\delta_2 > 0$
such that for $d' := \max(d,4)$ any $f \in C_b^{d'}(X)$ and any continuously differentiable compactly supported $\psi:\mathbb{R} \to \mathbb{R}$ we have that
\begin{align*}
    & \frac{1}{\phi(q)} \sum_{\substack{p \in \mathbb{Z}:\\ (p,q) = 1}} f(n(p/q)a(1/q))\psi(p/q) = \int_X f \; d \nu \int_{\mathbb{R}} \psi(t) \; dt \\
    + & O_{\epsilon}((W+1) C'(\xi,L)^{1/2} ||f||_{C^{d'}_b(X)} ||\psi||_{W^{1,\infty}} q^{- \delta_2+ \epsilon})
\end{align*}
where $L,W$ and $C'(\xi,L)$ are defined as in Corollary \ref{to compact support}.

\begin{proof}
We will suppose initially that supp$(\psi) \subset [a,b]$ with $b - a = 1$. We will also suppose, initially, that $f$ has zero mean. As in the proof of Lemma \ref{c independence}, let $F_q(t) := f(n(t)a(1/q))\psi(t)$.
Now, let 
$$
\hat{F}_q(k) := \int_a^b F_q(t)e(-kt) \; dt.
$$
By the standard theory of Fourier series, we have that
$$
F_q(t) = \sum_{k \in \mathbb{Z}} \hat{F}_q(k)e(kt)
$$
whenever $t \in [a,b]$. Using this, we see that
\begin{align*}
    \frac{1}{\phi(q)} \sum_{\substack{p \in \mathbb{Z}:\\ (p,q) = 1}} f(n(p/q)a(1/q))\psi(p/q) & = \frac{1}{\phi(q)} \sum_{k \in \mathbb{Z}} \hat{F}_q(k) \sum_{\substack{p \in [qa,qb]_{\mathbb{Z}}:\\ (p,q) = 1}} e(kp/q) \\
    & = \sum_{k \in \mathbb{Z}} \hat{F}_q(k) S(q,k)
\end{align*}
where $$S(q,k) := \frac{1}{\phi(q)}\sum_{\substack{p \in [qa,qb]_{\mathbb{Z}}:\\ (p,q) = 1}} e(kp/q) = \frac{1}{\phi(q)} \sum_{\substack{0 \leq k \leq q-1: \\(p,q) = 1}}e(kp/q).
$$

Since $S(q,k)$ clearly only depends upon $k \pmod{q}$, we can write 
\begin{equation}\label{S_T decomp}
    \sum_{k \in \mathbb{Z}} \hat{F}_q(k) S(q,k) = \sum_{r = 0}^{q-1} S(q,r) \sum_{k \equiv r \pmod{q}} \hat{F}_q(k).
\end{equation}

At this stage, let us note the two following identities which hold for all $m \in \mathbb{Z}$ and $q \in \mathbb{N}$:
\begin{equation}\label{rationalexpsum}
    \frac{1}{q} \sum_{p=0}^{q-1} e(mp/q) = \begin{cases}
    1 & \text { if } q|m \\
    0 & \text{ otherwise}
    \end{cases}
\end{equation}
and 
\begin{equation}\label{primitiveexpsum}
    \frac{1}{\phi(q)} \sum_{\substack{p=0 \\(p,q)=1}}^{q-1} e(mp/q) = \frac{\mu(q_m)}{\phi(q_m)}
\end{equation}
where $\mu$ is the Mobius function and $q_m := q/(q,m)$.

Using (\ref{rationalexpsum}) one sees that 
\begin{align*}
    \sum_{k \equiv r \pmod{q}} \hat{F}_q(k) & = \sum_{k \in \mathbb{Z}} \bigg(\frac{1}{q} \sum_{p=0}^{q-1}  
    e((k-r)p/q) \bigg) \\
    & = \frac{1}{q} \sum_{p=0}^{q-1} e(-rp/q) \sum_{k \in \mathbb{Z}} \hat{F}_q(k) e(kp/q) \\
    & = \frac{1}{q} \sum_{p \in \mathbb{Z}} F_q(p/q) e(-rp/q).
\end{align*}
But, we see applying Lemma \ref{c independence} that 
$$
\frac{1}{q} \sum_{p \in \mathbb{Z}} F_q(p/q) e(-rp/q) \ll C'(\xi,L)^{1/2} ||f||_{C^{d'}_b(X)} ||\psi||_{W^{1,\infty}} q^{- \delta_2}
$$
with implicit constant independent of $r$.

As a result, (\ref{S_T decomp}) tells us that 
$$
\sum_{k \in \mathbb{Z}} \hat{F}_q(k) S(q,k) \ll \frac{ C'(\xi,L)^{1/2} ||f||_{C^{d'}_b(X)} ||\psi||_{W^{1,\infty}}}{ q^{ \delta_2}} \sum_{r=0}^{q-1} |S(q,r)|.
$$

So, it remains to estimate $\sum_{r=0}^{q-1} |S(q,r)|.$
For any $m \in \mathbb{N}$ we find that
$$\sum_{r=0}^{m-1} \frac{1}{\phi(m_r)} = \sum_{d|m} \sum_{\substack{0 \leq r \leq m-1\\(r,m) = m/d}} \frac{1}{\phi(d)}  \leq
\sum_{d|m} \sum_{\substack{0 \leq r \leq m-1\\(m/d)|r}} \frac{1}{\phi(d)}.
$$
Thus, using that $\phi(d) \gg_{\epsilon} d^{1-\epsilon}$ and then that 
$\sum_{d|m} d^{\epsilon} \ll_{\epsilon} m^{2 \epsilon}$ we get that 
$$
\sum_{r=0}^{m-1} \frac{1}{\phi(m_r)} \ll_{\epsilon} m^{\epsilon}.
$$
Using (\ref{primitiveexpsum}) we therefore see that
$$
\sum_{r=0}^{q-1} |S(q,r)| \ll_{\epsilon} q^{\epsilon}.
$$
So, overall we have that
$$
\frac{1}{\phi(q)} \sum_{\substack{p \in \mathbb{Z}:\\ (p,q) = 1}} f(n(p/q)a(1/q))\psi(p/q) \ll_{\epsilon} \frac{ C'(\xi,L)^{1/2} ||f||_{C^{d'}_b(X)} ||\psi||_{W^{1,\infty}}}{ q^{ \delta_2- \epsilon}}
$$
To deal with the case in which $f$ has non-zero mean, we can perform the usual trick of replacing $f$ by $f_0 = f - \int_X f \; d \nu$. As a result, we are left to estimate 
$$
\sum_{\substack{p \in \mathbb{Z} : \\ (p,q) = 1}} \psi(p/q).
$$
which is equal to $\int \psi(t) \; dt + O_{\epsilon}(||\psi||_{W^{1,\infty}} q^{-1 + \epsilon}).
$
This is very similar to the above computation. To briefly outline, decomposing $\psi$ as a Fourier series on the interval $[a,b]$, we write
$\psi(t) = \sum_{k \in \mathbb{Z}} \hat{\psi}(k) e(kt)$ 
and using (\ref{primitiveexpsum}), we arrive at 
$$
\bigg| \sum_{\substack{p \in \mathbb{Z} : \\ (p,q) = 1}} \psi(p/q) - \int_X \psi(t) \; dt \bigg| \leq \sum_{k \in \mathbb{Z} \backslash \{0 \} } \frac{|\hat{\psi}(k)|}{\phi(q_k)} 
$$
Now, as above, we have that 
$$
\sum_{k \in \mathbb{Z} \backslash \{0 \} } \frac{|\hat{\psi}(k)|}{\phi(q_k)} \leq \sum_{d|q} \frac{1}{\phi(d)} \sum_{(q/d) | k} |\hat{\psi}(k)| = \frac{1}{q} \sum_{d|q} \frac{d}{\phi(d)} \sum_{k \in \mathbb{Z} \backslash \{0\} } \frac{1}{|k|} \bigg| \frac{kq}{d} \hat{\psi}\bigg( \frac{kq}{d} \bigg) \bigg| .
$$
Using integration by parts, we see know that $k |\hat{\psi}(k)| \asymp |\hat{\psi'}(k)|$ and so, applying the Cauchy-Schwartz inequality and Parseval's identity, we see that
$$
 \sum_{k \in \mathbb{Z} \backslash \{0\} } \frac{1}{|k|} \bigg| \frac{kq}{d} \hat{\psi}\bigg( \frac{kq}{d} \bigg) \bigg| \ll \sum_{k \in \mathbb{Z}} |\hat{\psi'}(k)|^2 \ll ||\psi'||_2
$$
The claimed bound then follows from applying the same bounds on $\phi(d)$ and the divisor sums we used previously. 

Finally, to deal with the case where the width of the support of $\psi$ is more than 1, we find a function $\Delta \in C_c^{\infty}(\mathbb{R})$ with supp$(\Delta) \subset [0,1]$ and which has the property that 
\begin{equation}\label{sumto1}
    \sum_{j \in \mathbb{Z}} \Delta(t- j/2) = 1
\end{equation}
for all $t \in \mathbb{R}$. Let $\Delta_j(t) := \Delta(t-j/2)$. Given that $\Delta_j \psi$ always has support of width at most 1 and that, by the product rule, 
$|| \psi \Delta_j ||_{W^{1,\infty}} \ll || \psi ||_{W^{1,\infty}} $ we have that 
\begin{align}
    & \frac{1}{\phi(q)} \sum_{\substack{p \in \mathbb{Z}:\\ (p,q) = 1}} f(n(p/q)a(1/q))\psi(p/q) \Delta_j(p/q) = \int_X f \; d \nu \int_{\mathbb{R}} \psi(t) \Delta_j(t) \; dt \label{onej}\\   +&  O_{\epsilon}\bigg(\frac{ C'(\xi,L)^{1/2} ||f||_{C^{d'}_b(X)} ||\psi||_{W^{1,\infty}}}{ q^{ \delta_2- \epsilon}} \bigg) \nonumber
\end{align}
Moreover, we have an exact equality here when the supports of $\Delta_j$ and $\psi$ don't intersect. Given the number of $j$ such that their supports do intersect is $\ll W + 1$ the result follows from summing (\ref{onej}) over all $j$ and using (\ref{sumto1}).

\end{proof}
\end{corollary}
\section{Concrete equidistribution results}\label{concreteEquidistribution}

To finish off, we prove the two main Corollaries stated in \S 1.

\begin{proof}[Proof of Corollary \ref{Strom Cor}]
In the case, Theorem 1.2 from \cite{strombergsson2015effective} tells us that, when $\xi = (\xi_1,\xi_2)$ is of Diophantine type $K$ with a constant $c >0$, (\ref{initialassumption}) holds with $C \geq 1$ being a fixed constant depending only on $(c,K)$, $\delta = \min(1/4,1/K)$ and $d = 8$. Moreover, here $A(\xi) = |\xi_1| + |\xi_2|$. Therefore, in this case we have that $C'(\xi,L) \leq C(c,K)(1 + L^{1/2} + |\xi_1| + |\xi_2| + 1)$ and so the claimed results follow from Corollaries \ref{to compact support} and \ref{to primitive compact support}.
\end{proof}

\begin{proof}[Proof of Corollary \ref{Brown Cor}]
In this case, Theorem 1.2 in \cite{browning2016effective} implies that (\ref{initialassumption}) holds with $d = 8$. 
Similarly to as in the proof of Corollary \ref{to primitive compact support}, we choose a function $\Delta \in C_c^{\infty}(\mathbb{R})$ with supp$(\Delta) \subset (-2,2)$ such that 
\begin{equation}\label{partitionof1}
    \sum_{j \in \mathbb{Z}} \Delta(t-2j) = 1
\end{equation}
for all $t \in \mathbb{R}$. Let $\Delta_j(t) := \Delta(t-2j)$. In the non-primitive case, we have that
\begin{align*}
    & \frac{1}{N} \sum_{k \in \mathbb{Z}} f(n(k/N)a(1/N))\psi(k/N) \\
    = & \sum_{j \in \mathbb{Z}} \frac{1}{N} \sum_{k \in \mathbb{Z}} f(n(k/N)a(1/N))\psi(k/N) \Delta(k/N - 2j)
\end{align*}
Using that $n(t)$ has period 2, Corollary \ref{to compact support}, and that $||\psi \Delta_j||_{W^{1,\infty}} \ll ||\psi||_{W^{1,\infty}}$ we have that
\begin{align}
    & \bigg| \frac{1}{N} \sum_{k \in \mathbb{Z}} f(n(k/N)a(1/N))\psi(k/N) \Delta(k/N - 2j)  - \int_X f \; d \nu \int_{\mathbb{R}} \psi \Delta_j \; dt \bigg| \label{summableoverj} \\
    = & \bigg| \frac{1}{N} \sum_{k \in \mathbb{Z}} f(n(k/N)a(1/N))\psi(k/N+2j)  \Delta(k/N)-  \int_X f \; d \nu \int_{\mathbb{R}} \psi \Delta_j \; dt \bigg| \nonumber \\
    \leq & C ||f||_{C^8_b(X)}||\psi||_{W^{1,\infty}} N^{-\delta_1} \label{boundonsummableoverj}
\end{align}
for some constant $C > 0$. By summing (\ref{summableoverj}) over all $j$ such that the support of $\Delta_j$ and $\psi$ 
intersect, of which there are $\ll (W+1)$, using the bound
(\ref{boundonsummableoverj}) and (\ref{partitionof1}), we get the result in the non-primitive case. 

The result in the primitive case uses the exact same argument with the sum over $k$ replaced by a sum over all $p$ with $(p,q)=1$ and so we omit the details. 
\end{proof}

\section{Acknowledgements}

I would like to thank Jens Marklof for his advice throughout the 
writing of this paper and the Heilbronn Institute for Mathematical Research for their support.
\printbibliography

\end{document}